\documentclass[11pt,oneside]{amsart}
\usepackage[utf8]{inputenc}

\usepackage[english]{babel}

\usepackage[dvips]{graphics,epsfig}
\usepackage{times}
\usepackage{amsmath,amsthm,amssymb}
\usepackage{array}
\usepackage{bbm}
\usepackage{xcolor}
\usepackage{verbatim}

\usepackage{todonotes}

\usepackage{geometry}

\def\AAA{\mathbb{A}}

\def\CC{\mathcal{C}}

\def\NNN{\mathbb{N}}

\def\QQQ{\mathbb{Q}}

\def\UU{\mathcal{U}}

\def\QQQa{\mathbb{Q}\langle \Bar{\alpha},\alpha_0 \rangle}

\def\a{\alpha}
\def\b{\beta}

\def\d{\delta}

\def\D{\Delta}

\def\nmdeg{\operatorname{nmdeg}}

\DeclareMathOperator{\RU}{RU}
\DeclareMathOperator{\RM}{RM}

\DeclareMathOperator{\tp}{tp}

\DeclareMathOperator{\ord}{ord}
\DeclareMathOperator{\trdeg}{trdeg}

\def\Ind#1#2{#1\setbox0=\hbox{$#1x$}\kern\wd0\hbox to 0pt{\hss$#1\mid$\hss}
\lower.9\ht0\hbox to 0pt{\hss$#1\smile$\hss}\kern\wd0}
\def\ind{\mathop{\mathpalette\Ind{}}}
\def\notind#1#2{#1\setbox0=\hbox{$#1x$}\kern\wd0
\hbox to 0pt{\mathchardef\nn=12854\hss$#1\nn$\kern1.4\wd0\hss}
\hbox to 0pt{\hss$#1\mid$\hss}\lower.9\ht0 \hbox to 0pt{\hss$#1\smile$\hss}\kern\wd0}

\newtheorem{theorem}{Theorem}[section]
\newtheorem{thm}[theorem]{Theorem}
\newtheorem{lemma}[theorem]{Lemma}
\newtheorem{cor}[theorem]{Corollary}
\newtheorem{prop}[theorem]{Proposition}
\newtheorem{fact}[theorem]{Fact}
\newtheorem{claim}[theorem]{Claim}
\newtheorem{conj}[theorem]{Conjecture}

\theoremstyle{definition}
\newtheorem{definition}[theorem]{Definition}
\newtheorem{example}[theorem]{Example}

\newtheorem{ques}[theorem]{Question}

\theoremstyle{remark}
\newtheorem{remark}[theorem]{Remark}
\newcommand{\m}{\mathbb }
\newcommand{\mc}{\mathcal }

\begin{document}

\title{Generic differential equations are strongly minimal}
\author[M. DeVilbiss]{Matthew DeVilbiss}
\author[J. Freitag]{James Freitag}
\address{Matthew DeVilbiss, Ohio State University, Department of Mathematics, 231 W 18th Ave, Columbus, OH, USA, 43210-1174.}
\email{devilbiss.16@osu.edu}
\address{James Freitag, University of Illinois Chicago, Department of Mathematics, Statistics,
and Computer Science, 851 S. Morgan Street, Chicago, IL, USA, 60607-7045.}
\email{jfreitag@uic.edu}
\thanks{The authors were partially supported by NSF CAREER award 1945251 during the course of this work. The authors thank Dave Marker, Ronnie Nagloo, Anand Pillay, and especially Rahim Moosa for useful conversations around this work. The techniques developed in this paper build on the thesis of Jonathan Wolf cited below. The authors also wish to acknowledge the extremely detailed and useful remarks of the referee.}


\maketitle
\begin{abstract}
In this manuscript we develop a new technique for showing that a nonlinear algebraic differential equation is strongly minimal based on the recently developed notion of the degree of nonminimality of Freitag and Moosa. Our techniques are sufficient to show that generic order $h$ differential equations with nonconstant coefficients are strongly minimal, answering a question of Poizat (1980).     
\end{abstract}

\section{Introduction} 
Let $f(x)=0$ be an algebraic differential equation of in a single indeterminate $x$ with coefficients in a differential field $(K, \delta)$ of characteristic zero. In this manuscript, we are particularly interested in the case that $f(x)$ is nonlinear and of order $\geq 2$. The central property we study is the \emph{strong minimality} of the solution set of $f(x)=0$. The notion of strong minimality comes from model theory; in general, a \emph{definable set} $X$ is \emph{strongly minimal} if every definable subset is finite or cofinite, uniformly in parameters. In our setting, we are interested in the situation $X = \{ x \in \mc U \, | \, f(x) =0 \}$ is the set of solutions to an algebraic differential equation where $\mc U$ is a differentially closed field. Let $h$ be the order of $f$ -- that is, the highest derivative of $x$ appearing in $f$. In this context, $X$ is strongly minimal if and only if the multivariate polynomial $f$ is irreducible over $K^{alg}$ and given any $a \in \mc U $ with $f(a) = 0,$ and any differential field $ K_1 \leq \mc U$ with $K \leq K_1$, the transcendence degree of $K_1 \langle a \rangle $ over $K_1$ is either $0$ or $h$.

Strong minimality is an intensively studied property of definable sets, and has been at the center of many important number theoretic applications of model theory and differential algebra \cite{casale2020ax, freitag2017strong, HrushovskiMordell-Lang, nagloo2017algebraic}. Despite this, there are relatively few (classes of) equations which have been shown to satisfy the property - so few, that we are in fact able to give below what we believe to be a (at the moment) comprehensive list of those equations which have been shown to be strongly minimal. Showing the strong minimality of a given equation is itself sometimes a motivational goal, but often it is an important piece of a more elaborate application, since it allows one to use powerful tools from geometric stability theory. The existing strategies to prove strong minimality are widely disparate but apply only to very special cases. In roughly chronological order:  
\begin{enumerate}
\item Poizat established that the set of non-constant solutions of $x \cdot x'' =x'$ is strongly minimal (see \cite{MMP} for an explanation). Poizat's arguments were generalized by Brestovski \cite{brestovski1989algebraic} to a class of very specifically chosen order two differential equations with constant coefficients. 

\item Hrushovski's work \cite{HrushovskiMordell-Lang} around the Mordell-Lang conjecture proved the strong minimality of Manin kernels of nonisotrivial simple abelian varieties. It uses specific properties of abelian varieties as well as model-theoretic techniques around \emph{modularity} of strongly minimal sets. 
    
\item Nagloo and Pillay \cite{nagloo2017algebraic} show that results of the Japanese school of differential algebra \cite{murata1995classical, ohyama2006studies, okamoto1986studies, okamoto1987studies5, okamoto1987studies, umemura1997solutions, umemura1998solutions, watanabe1995solutions, watanabe1998birational} imply that Painlev\'e equations with generic coefficients are strongly minimal. The techniques employed are differential algebraic and valuation theoretic, relying on very specific properties of the equations. 

\item Work of Freitag and Scanlon \cite{freitag2017strong} shows that the differential equation satisfied by the $j$-function is strongly minimal. This result ultimately relies on point-counting and o-minimality via the Pila-Wilkie theorem as applied in \cite{PilaAO, PilaDer}; the argument there is very specific to the third order nonlinear differential equation satisfied by the $j$-function. Later, Aslanyan \cite{aslanyan2020ax} produced another proof, ultimately relying on similar (stronger) inputs of \cite{pila2016ax}. 
    
\item Casale, Freitag and Nagloo \cite{casale2020ax} show that equations satisfied by $\Gamma$-automorphic functions on the upper half-plane for $\Gamma$ a Fuchsian group of the first kind are strongly minimal. The arguments use differential Galois theory with some additional analytic geometry, and the techniques again are very specific to the third order equations of this specific form.

\item Jaoui shows that generic planar vector fields over the constants give rise to strongly minimal order two differential varieties \cite{jaoui2019generic}. The arguments rely on various sophisticated analytic techniques and results from foliation theory, some of which are particular to the specific class of equations considered.

\item Bl{\'a}zquez-Sanz, Casale, Freitag, and Nagloo \cite{blazquez2020some} prove the strong minimality of certain general Schwarzian differential equations.

\item Freitag, Jaoui, Marker, and Nagloo \cite{freitag2022equations} show that various equations of Li\'enard-type are strongly minimal using techniques from valuation theory. 
\end{enumerate}

We should also mention that strong minimality in this context was perhaps first studied by Painlev\'e using different language in \cite{painleve1leccons}. Painlev\'e conjectured the strong minimality of various classes of differential equations, where the notion is equivalent to \emph{Umemura's Condition (J).} See \cite{nagloo2017algebraic} for a discussion of these connections. We believe that the above list, together with a specific example 
of \cite{freitag2012model} constitutes the entire list of differential equations (of order at least two) which have been proven to be strongly minimal. Most of the techniques in the above listed results apply only to specific equations or narrow classes of equations and rely on specific properties of those classes in proving strong minimality. Our goal in this article will be to develop a rather more general approach which applies widely to equations with at least one differentially transcendental coefficient.

\subsection{Our approach and results} 
Let $f \in k\{ x \}$. Generally speaking, when attempting to prove strong minimality\footnote{Equivalently, there are no infinite differential subvarieties.} of some differential variety $$V = Z(f)=\{a \in \mc U \, | \, f(a) =0 \},$$ there are two phenomena which make the task difficult:
\begin{enumerate} 
\item There is no a priori upper bound on the degree of the differential polynomials which define a differential subvariety of $V$. 
\item The differential polynomials used to define a differential subvariety might (necessarily) have coefficients from a differential field extension of the field of $k$. 
\end{enumerate} 
There are structure theorems related to (1) but only in special cases. See for instance \cite{freitag2017finiteness} when the subvarieties are co-order one in $V$. Controlling the field extension in (2) is a key step in various recent works \cite{freitag2017strong, jaoui2019generic, nagloo2017algebraic}. This is most often accomplished by noting that \emph{stable embeddedness} of the generic type of $V$ implies that the generators of the field of definition of a forking extension can be assumed to themselves realize the generic type of $V$ -- see explanations in \cite{casale2020ax, freitag2017strong}. In recent work, Freitag and Moosa \cite{freitagmoosa} introduce a new invariant of a type, which more closely controls the structure over which the forking extension of a type is defined:

\begin{definition}
Suppose $p\in S(A)$ is a stationary type of $U$-rank $>1$.
By the {\em degree of nonminimality} of $p$, denoted by $\nmdeg(p)$, we mean the least positive integer $k$ such that for a Morley sequence of $p$ of length $k$, say $(a_1,\dots,a_k)$, $p$ has a nonalgebraic forking extension over $A,a_1,\dots,a_k$.
If $\RU(p)\leq 1$ then we set $\nmdeg(p)=0$. 
\end{definition}

In the theory of differentially closed fields of characteristic zero, Freitag and Moosa \cite{freitagmoosa} give an upper bound for the degree of nonminimality in terms of Lascar rank:

\begin{thm} \label{boundMorley}
Let $p \in S(k)$ have finite rank. Then $\nmdeg(p) \leq \RU(p)+1.$
\end{thm}

Let $a \models p$, we will call the transcendence degree of the differential field $k\langle a \rangle /k $ the order of $p$. When $p$ is the generic type of a differential variety $V$, we also call this the order\footnote{These notions of order agree with the previously mentioned definition of order for differential polynomials.} of $V$. The order of $p$ is an upper bound for the Morley rank of $p$. The Morley rank of $p$ is a bound for the Lascar rank of $p$. For proofs of these facts, see \cite{MMP}. It follows that \emph{if} the type $p$ of a generic solution of an order $n$ differential equation over $k$ has a nonalgebraic forking extension over some differential field extension, then already $p$ has such a forking extension over $k\langle a_1, \ldots , a_{n+1} \rangle$ where the $a_i$ are from a Morley sequence in the type of $p$ over $k$. This consequence of Theorem \ref{boundMorley} will be essential to our approach to handling issue (2) above. 

Our approach to issue (1) follows a familiar general strategy of reducing certain problems for nonlinear differential equations to related problems for associated linear differential equations. For instance, \cite{PillayZiegler} applies a strategy of this nature to establish results around the Zilber trichotomy, while \cite{casale2020ax,nagloo2019algebraic} use this strategy to establish irreducibility of solutions to automorphic and Painlev\'e equations using certain associated Riccati equations. Our technique fits into this general framework and relies on Kolchin's differential tangent space, which will provide the linear equations associated with the original nonlinear differential variety $V$. Our approach to the associated linear equations has been under development in the thesis of Wolf \cite{wolf2019model} and the forthcoming thesis of DeVilbiss which gives an approach to calculating the Lascar rank of underdetermined systems of linear differential equations.  

When we call a differential polynomial $f(x)$ \emph{generic} of order $h$ and degree $d$, if $f$ is a linear combination of all monomials of degree no more than $d$ on the variables $x, x' ,\ldots, x^{(h)}$ with independent differentially transcendental coefficients.

Our main theorem is: 
\begin{thm}\label{thm:main}
Let $f(x)$ be a generic differential polynomial of order $h>1$ and degree $d$. Let $p$ be the type of a generic solution to $Z(f)$. If $d \geq 2 \cdot (\nmdeg(p)+1),$ then $Z(f)$ is strongly minimal.
\end{thm} 
Since $\nmdeg(p)\leq \RU(p) +1 \leq \ord(f)+1$, the following corollary is immediate:
\begin{cor}\label{cor:main}
Let $f(x)$ be a generic differential polynomial of order $h>1$ and degree $d$. If $d \geq 2 \cdot (h+2),$ then $Z(f)$ is strongly minimal. 
\end{cor}

This answers Question 7 of \cite{poizat1980c} for sufficiently large degree, any order, and nonconstant coefficients. As described above, Jaoui \cite{jaoui2019generic} has recently answered the order two case of Question 7 of \cite{poizat1980c} for constant coefficients. We conjecture a more general form of Theorem \ref{thm:main}: 

\begin{conj}
Generic differential equations of fixed order and degree greater than one are strongly minimal. 
\end{conj}

To the non-model theorist, it does not seem obvious why strong minimality plays a central role in the theory of algebraic differential equations, but there seem to be two important factors behind this: \begin{itemize}
\item Once strong minimality of an equation is established, powerful results having their origins in geometric stability theory can be employed (e.g. the Zilber trichotomy, discussed next). 
\item Among nonlinear equations, the results of this paper and those of \cite{jaoui2019generic} show that strong minimality holds ubiquitously.  
\end{itemize} 

Even when an equation is not strongly minimal, it is often true that questions about the solutions can be reduced to questions about solutions of certain associated minimal equations coming from a notion called \emph{semi-minimal analysis}; for instance, see \cite{freitag2022any}. We now state the Zilber trichotomy for strongly minimal sets adapted to the setting of differentially closed fields:

\begin{fact}[\cite{HrSo},\cite{PillayZiegler}]\label{trichotomy} Let $X$ be a strongly minimal set. Then exactly one of the following holds:
\begin{enumerate}
\item (non-locally modular) $X$ is nonorthogonal to $\mathbb{C}$,
\item (locally modular, nontrivial) $X$ is nonorthogonal to the (unique) smallest Zariski-dense definable subgroup of a simple abelian variety $A$ which does not descend to $\mathbb{C}$, 
\item (trivial) $X$ is geometrically trivial.
\end{enumerate}
\end{fact}
Nonorthogonality is a natural equivalence relation on strongly minimal sets. For strongly minimal sets $X$ and $Y$, $X$ is nonorthogonal to $Y$ if and only if there is a generic finite-to-finite definable correspondence between $X$ and $Y$. A strongly minimal set is geometrically trivial if whenever $a_1, \ldots , a_n \in X$ are are dependent in the sense of forking, there is a pair of distinct $a_i, a_j$ which is dependent. In terms of algebraic relations in our setting, anytime $a_1, \ldots , a_n \in X$ are generic solutions of the differential equation $X$, if $$\trdeg _K (K \langle a_1, \ldots a_n \rangle  ) < n \ord (X),$$ then there is a distinct pair $a_i,a_j$ such that 
$$\trdeg _K (K \langle a_i,a_j \rangle  ) < 2 \ord (X).$$
By strong minimality, the previous inequality implies that
$$\trdeg _K (K \langle a_i,a_j \rangle  ) = 0\text{ or } \ord (X).$$
While Fact \ref{trichotomy} gives a rather complete classification of strongly minimal sets in the nontrivial cases of the Zilber trichotomy, there is no general classification of the geometrically trivial strongly minimal sets in differentially closed fields. 

We conjecture that generic differential equations are \emph{geometrically trivial} in a strong form: 

\begin{conj}
Any two solutions of a generic differential equation of fixed order and degree greater than one are (differentially) algebraically independent. 
\end{conj}

In this paper, our techniques are applied to equations with differentially transcendental coefficients, but this is not an inherent restriction of the methods. For instance, in forthcoming work using these techniques joint with Casale and Nagloo, we give a fundamentally new proof of the main theorem of \cite{freitag2017strong}, proving that the equation satisfied by the $j$-function is strongly minimal. See Section \ref{genstrongmin1} for additional discussion. 

\subsection{Organization}
In section \ref{Notation}, we set up the notation and background results we require. Section \ref{Lintrans} gives a new sufficient condition for the strong minimality of a differential variety. Section \ref{genstrongmin} applies this condition to show that generic differential equations are strongly minimal. Section \ref{genstrongmin1} shows how one can establish a weaker condition than strong minimality in a more computationally straightforward manner and gives some open problems.

\section{Notation} \label{Notation} 
Let $\mc U$ be a countably saturated differentially closed field of characteristic zero and let $\CC$ be the field of constants of $\UU$. All of the fields we consider will be subfields of $\mc U$. An affine \emph{differential variety} is the zero set of a (finite) system of irreducible differential polynomial equations over (a finitely generated subfield of) $\mc U$. 

Let $(y_1, \ldots, y_n)$ be a finite set of differential indeterminants over $\mc U$ and let $\Theta$ denote the set of derivative operators on $\mc U$. Since we are interested in differential fields with a single derivation, $\Theta=\{\d^{k}:k\geq 0\}$. A \emph{ranking}  on $(y_1, \ldots y_n)$ is a total ordering on the derivatives $\{\theta y_j: \theta \in \Theta, 1\leq j \leq n\}$ such that for all such derivatives $u$, $v$, and all $\theta\in \Theta$, we have
$$
u\leq \theta u, \quad u\leq v \Rightarrow \theta u \leq \theta v.
$$
A ranking is \emph{orderly} if whenever the order of $\theta_1$ is lower than the order of $\theta_2$, we have $\theta_1 y_i < \theta_2 y_j$ for any $i, j$.
An \emph{elimination ranking} is a ranking in which $y_i < y_j$ implies $\theta_1 y_i < \theta_2 y_j$ for any $\theta_1,\theta_2\in \Theta$.
For a $\d$-polynomial $f(y_1, \ldots, y_n)$, the highest ranking $\theta y_j$ appearing in $f$ is the \emph{leader} of $f$, denoted $u_f$. If $u_f$ has degree $d$ in $f$, we can rewrite $f$ as a polynomial in $u_f$, $f= \sum_{i=0}^d I_i u_f^i$, where the \emph{initial} of $f$, $I_d$, is not zero.
The \emph{separant} of $f$ is the formal derivative $\frac{\partial f}{\partial u_f}$.
Note that the leader, initial, and separant of a differential polynomial are defined only after choosing a specific ranking. A detailed treatment of these definitions can be found in \cite[pg 75]{KolchinDAAG}.

Let $\bar a \in \mc U$, and $F$ a differential subfield of $\mc U$. There is a numerical polynomial $\omega (\bar a /F)$ called the \emph{Kolchin polynomial} of $\bar a$ over $F$ such that for sufficiently large $s\in \NNN$,
$$\omega_{\bar{a}/F}(s)=\trdeg\left(F\left(\bar{a}, \d(\bar{a}), \ldots, \d^s(\bar{a})\right) / F\right)$$
(see \cite[Theorem 6, pg 115]{KolchinDAAG}). When $X$ is a differential variety, that is, a closed irreducible set in the Kolchin topology, $\omega (X/F):=\omega (\bar a /F)$ where $\bar a $ is a generic point on $X$ over $F$.
Since we are only concerned with differential fields with a single derivation, the degree of $\omega(X/F)$, sometimes called the \emph{differential type} is either 0 or 1. The leading coefficient of $\omega(X/F)$ is called the \emph{typical differential dimension}. Crucially for our purposes, the Kolchin polynomial witnesses forking in the sense of model theory , i.e., 
$$ \bar{a} \ind_A C \Longleftrightarrow \omega_{\bar{a}/A}(s)=\omega_{\bar{a}/A\cup C}(s).$$
Proof of this fact can be found in Theorem 4.3.10 of \cite{mcgrail2000model} and in Proposition 2.8 of \cite{pong2003rank}.

Let $X\subset \mathcal{U}^n$ be a differential variety and let $I(X)$ be the ideal of differential polynomials vanishing on $X$. The \emph{differential tangent bundle} $T^\Delta X$ is the differential variety given by
$$ \left\{ (x_1,\ldots, x_n, w_1, \ldots, w_n)\in \UU^{2n}: \forall f\in I(X), f(\bar{x})=0, \sum_{\substack{i\leq n\\ j\leq \ord(f)}} \frac{\partial f}{\partial x_i^{(j)}}(\bar{x})w_i^{(j)}=0\right\}$$
where $\frac{\partial f}{\partial x_i^{(j)}}(\bar{x})$ is a formal partial derivative where $x_i^{(j)}$ are considered as algebraic variables.
Given $\bar{a}\in X$, the \emph{differential tangent space} of $X$ over $\bar{a}$, denoted $T_{\bar{a}}^\Delta X$ is the fiber of $T^\Delta X$ over $\bar{a}$.
We now state a few basic properties of $T^\D X$ which will be used later:
\begin{itemize}
    \item The fibers $T_{\bar a}^\D X$ are (possibly infinite dimensional) $\CC$-vector spaces where $\CC$ is the field of constant elements in $\UU$.
    \item If $Y$ is a closed differential subvariety of $X$ and $\bar a\in Y$, then $T_{\bar a}^\D Y$ is a closed $\CC$-vector subspace of $T_{\bar a}^\D X$.
\end{itemize}
The following lemma appears as a corollary of \cite[Theorem 1, pg 199]{KolchinDAG}. The sentence involving a single differential variable is an improvement on Kolchin's corollary and follows immediately from Kolchin's proof.
\begin{lemma}\label{kolpoly}
Let $F$ be a differential field, $X$ a differential variety defined over $F$. Then there is a Kolchin-open set $U\subseteq X$ such that for every $\bar{a}\in U$, the Kolchin polynomial $\omega \left(X/F \right) = \omega \left( T^\D_{\bar{a}} (X) / F\langle \bar{a} \rangle \right)$. Moreover, if $X$ is defined by the vanishing of a single differential polynomial $f\in F\{x_1, \ldots, x_n\}$, then $U$ can be taken to be the open set defined by $I_fs_f\neq 0$ where $I_f$ is the initial of $f$ and $s_f$ is the separant of $f$ with respect to some orderly ranking.
\end{lemma}

\section{A general sufficient criterion for strong minimality} \label{Lintrans}
Let $f(x)$ be an order $n\geq 1$ non-linear differential polynomial in one variable without a constant term.
Let $\Bar{\a}$ denote the coefficients of $f$ and let $\a_0$ be differentially transcendental over $\Bar{\a}$.
Let $V_0$ be the differential variety corresponding to $f(x)=\a_0$.
Our goal in this section is to find sufficient conditions under which such a variety $V_0$ is strongly minimal.

Our next proposition shows that when $\a_0$ is differentially transcendental over $\QQQ\langle \bar{\a}\rangle$, there are no proper subvarieties of $V_0$ which are defined over the field $\m Q \langle \bar \a , \a_0 \rangle .$ Though the argument is simple, an elaboration of the technique in the proof will be used in the more difficult general case where one extends the field of coefficients. 

\begin{prop}\label{mindefn}
Let $f(x)$ be a non-linear order $h\geq 1$ differential polynomial with coefficients $\bar{\a}$, let $\a_0$ be differentially transcendental over $\QQQ\langle\bar{\a}\rangle$, and let $V_0$ be the differential variety defined by $f(x)=\a_0$. Then $V_0$ has no infinite proper subvarieties that are defined over $\QQQ\langle\Bar{\a},\a_0\rangle$.
\end{prop}
\begin{proof}
Suppose towards a contradiction that $W_0$ is an infinite proper subvariety of $V_0$ defined over $\QQQ\langle \Bar{\a}, \a_0 \rangle$. Then $W_0$ is given by some positive order $\d$-polynomial $g(x)\in\QQQ\langle \Bar{\a}, \a_0 \rangle \{x\}$. By clearing the denominators of $\alpha_0$, we can write $g(x,\a_0) \in \QQQ\langle \Bar{\a}\rangle\{x, \a_0\}$. For ease of notation, let $k=\QQQ\langle \Bar{\a}\rangle$.

Let $V$ be the differential variety given by $f(x)=y$ and let $W$ be given by $g(x,y)=0$ so that each instance of $\a_0$ is replaced with the variable $y$. These varieties are now defined by $\d$-polynomials in two variables with coefficients in $k$ and
$W \subsetneq V$. Let $a=(a_1,a_2)$ be a generic point of $W$ over $k$. Since $\alpha_0$ is differentially transcendental over $\bar \alpha$ the locus of $y$ over $k$ is $\m A^1$, so it follows that $W$ is an infinite rank (proper) subvariety of $V$. Consider the orderly ranking with $x$ ranked higher than $y$. 

We claim that the generic point $a$ of $W$ lies outside the locus on $V$ where the separant of $f(x)-y$ vanishes (we will call this the singular locus of $V$). This follows because the locus of the separant of $f$ inside of $V$ is finite rank (to see this, note that the separant is a differential polynomial in $k\{x\}$ so its generic solution has $x$-coordinate differentially algebraic over $k$). 
From the fact that $a$ lies outside the singular locus of $V$ and the singular locus of $W$ (since $a$ is generic on $W$), it follows that the Kolchin polynomials of $T^\D_{a} W$ and $T^\D_{a} V$ are equal to the Kolchin polynomials of $W$ and $V$, respectively, and so $T^\D_{a} W \subsetneq T^\D_{a} V$.

For $0\leq i \leq n$, let
\[
\b_i(x)  = \frac{\partial f}{\partial x^{(i)}}(x)
\]
denote the formal derivative of $f$ with respect to the $i$th derivative of $x$.
Using this notation, the differential tangent space $T^\D_a V $ is the set of $(w,z)$ satisfying the linear differential equation
\[
z=\sum_{i=0}^n \b_i(a)w^{(i)}.
\]
From this equation, we can see that $z$ is determined by our choice of $w$, but $w$ may be chosen freely. This gives a definable bijection between $T^\D_a V$ and $\AAA^1(\UU)$.
Further, it follows that $T^\D_a V$ has no infinite rank subspaces over $k\langle a \rangle$, since if it did, we could consider the image of this subvariety under the definable bijection to $\AAA^1(\UU)$. However, $\AAA^1(\UU)$ has no infinite rank subsets, so the image must have finite rank. Therefore, $\omega \left(T^\D_a W /k\langle a \rangle\right)$ is finite, a contradiction. 
\end{proof}

\begin{remark}\label{defnbijremark}
For $X$ a differential variety over a differential field $\QQQ\langle \bar{\a}\rangle$, the following conditions on the differential tangent space $T_{\bar{a}}^\D X$ are equivalent:
\begin{enumerate}
    \item The Lascar rank, $\RU\left(T^\Delta_{\bar{a}} V\right)=\omega$.
    \item The Morley rank, $\RM\left(T^\Delta_{\bar{a}} X\right)=\omega$.
    \item The differential tangent space $T^\Delta_{\bar{a}} X$ has no proper infinite rank $\CC$-vector subspaces definable over $\QQQ\langle \bar{a},\bar{\a}\rangle$.
    \item The differential tangent space $T^\Delta_{\bar{a}} X$ has no proper infinite rank $\CC$-vector subspaces definable over $\UU$.
    \item The differential tangent space $T^\Delta_{\bar{a}} X$ has no proper infinite rank subvarieties definable over $\QQQ\langle \bar{a},\bar{\a}\rangle$.
    \item The differential tangent space $T^\Delta_{\bar{a}} X$ has no proper infinite rank subvarieties definable over $\UU$.
\end{enumerate}
The equivalence of (1) and (2) is proved in \cite{PillayPong} and the equivalence of the others follows from the Berline-Lascar decomposition \cite[Theorem 6.7]{Poizat}.
By the argument at the end of the previous proof, each of these properties is implied by the existence of a definable bijection between the differential tangent space $T^\Delta_{\bar{a}} X$ and $\AAA^1(\UU)$.
\end{remark}

\begin{remark}
The previous result shows that under very general circumstances, for instance when any single coefficient is differentially transcendental over the others, the equation has no subvarieties over the coefficients of the equation itself. We state the following result, but omit its proof, as it is analogous to the previous proof and will not be used later in this paper. 
\end{remark}

\begin{prop}
Let $f$ be a differential polynomial in one variable and $V$ the zero set of $f$. Let $\bar \a$ denote the tuple of coefficients in $f$. If $\bar{\a}$ has some element ${\a_1}$ such that $\a_1$ is differentially transcendental over $\m Q \langle \bar \a _{-1} \rangle$,\footnote{By $\bar \a _{-1}$, we mean the tuple $\bar{\a}$ excluding $\a_1$.} then $V$ has no differential subvarieties over $\m Q \langle \bar \a \rangle$ except perhaps the zero set given by the monomial of which $\a_1$ is a coefficient. 
\end{prop}

The previous proposition works in such generality, in part because we have restricted the coefficient field. In various situations, identifying differential subvarieties defined over the field of definition of a variety $V$ is a \emph{much easier problem} than identifying differential subvarieties of $V$ defined over differential field extensions. For instance, in \cite{Nishioka}, Nishioka shows that the equations corresponding to automorphic functions of dense subgroups of $SL_2$ have to differential subvarieties over $\m C $. In the special case of genus zero Fuchsian functions, a much more difficult argument was required to extend the result to differential subvarieties over differential field extensions \cite{casale2020ax}, answering a long-standing open problem of Painlev\'e. 

There is one general purpose model theoretic tool which restricts the field extensions one needs to consider. We will use a principle in stability theory, generally related to \emph{stable embeddedness} (see for instance see \cite{ChHr-AD1} where this general type of result is referred to as the \emph{Shelah reflection principle}). For the following result see Lemma 2.28 \cite{GST}: 

\begin{lemma} \label{morseq}
In a superstable theory, let $A\subseteq B$ and $p\in S(B)$ which forks over $A$. Then there is an indiscernible sequence $(a_i:i\in \NNN)$ such that the canonical base of $p$ is contained in the definable closure of $A, a_1, \ldots , a_d$ where $a_1, \ldots, a_d$ is a finite initial segment of this indiscernible sequence.
\end{lemma}

Let $f$ and $V_0$ be as before and let $d\in \NNN$.
Consider $V_0^{[d]}$, the set of $d$-tuples so that each coordinate $x_i$ satisfies $f(x_i)=\a_0$.
As before, we can replace each instance of $\a_0$ with a new variable $y$, resulting in a differential variety $V^{[d]}$ defined by the system of equations:
\[
\left\{
\begin{array}{lcl}
f(x_1)  & = & y \\
f(x_2)  & = & y \\
  & \vdots &  \\
f(x_d)  & = & y 
\end{array}
\right.
\]

\begin{prop} \label{linnonlin}
Let $f(x)$ be a non-linear order $h\geq 1$ differential polynomial with coefficients $\bar{\a}$, let $\a_0$ be differentially transcendental over $\QQQ\langle\bar{\a}\rangle$, and let $V_0$ be the differential variety defined by $f(x)=\a_0$. Suppose that for all $d\in \NNN$ and for all indiscernible sequences $\bar{a}=(a_1,\ldots a_d)$ in the generic type of $V_0$ with $a_1, \ldots a_d$ algebraically independent over $\QQQ\langle \bar{\a}\rangle$, we have that the differential tangent space $T^\D_{\bar{a}}\left( V^{[d]} \right)$ has no proper infinite rank $\CC$-vector subspaces definable over $\QQQ\langle \bar{\a}, \bar{a}\rangle$.
Then $V_0$ is strongly minimal.
\end{prop}

\begin{proof}
Suppose $V_0$ is not strongly minimal and let $p(x)\in S_1\left(\QQQ\langle \Bar{\a},\a_0 \rangle \right)$ be the type of a generic solution of $V_0$. 
By Proposition \ref{mindefn}, $V_0$ does not have any infinite subvarieties defined over $\QQQ\langle \Bar{\a},\a_0 \rangle$, so $p$ has a forking extension $q$ over a differential field extension $K>\QQQ\langle \Bar{\a},\a_0 \rangle$.
By Lemma \ref{morseq}, there is some finite  $d$ and a Morley sequence $(a_1, \ldots , a_d)$ for $q$ such that $(a_1, \ldots , a_d)$ is not $\QQQa$-independent. Consider the minimal such $d$. 
Then $\tp\left(a_1/ \mathbb Q \langle \bar \alpha, \alpha_0, a_2, a_3, \ldots , a_d \rangle  \right)$ forks over $\QQQa$.
Since these are types over differential fields, this happens exactly when the Kolchin polynomial of $\left(a_1/ \mathbb Q \langle \bar \alpha, \alpha_0, a_2, a_3, \ldots , a_d \rangle  \right)$ is strictly less than the Kolchin polynomial of $\left(a_1/\QQQa\right)$ because the Kolchin polynomial witnesses forking.

Thus, there is a differential polynomial $g(x)\in \QQQ\langle \bar{\a},\a_0,a_2, \ldots , a_d\rangle \{x\}$ so that $g(a_1)=0$ and $g$ has order strictly less than the order of $f$.
By clearing denominators, we can write $g(x_1, \ldots , x_d ) \in\QQQ \langle \bar{\a}, \a_0 \rangle \{x_1 , \ldots , x_d\}$ such that $g(a_1, \ldots , a_d)=0$.
Let $U_0\subset V_0^{[d]}$ be the vanishing set of $g(x_1,\ldots , x_d)$.
Just as with $V$, we can replace $\a_0$ with a new variable $y$ after clearing denominators again, giving a $\QQQ\langle \bar{\a}\rangle$-polynomial $g(x_1,\ldots , x_d, y)$ and the corresponding variety $U\subset V^{[d]}$.
The Kolchin polynomial $\omega \left( U / \QQQ\langle \bar{\a}\rangle\right)$ is nonconstant.
Let $\bar{a}=(a_1, \ldots , a_d, \a_0)$ and notice that $\bar{a}$ is a generic point of $U$ over $\QQQ\langle \a \rangle$.
By Lemma \ref{kolpoly}, the Kolchin polynomial of the differential tangent space
$\omega \left( T^\Delta_{\bar{a}} (U)/\QQQ\langle \bar{\a}, \bar{a}\rangle \right)$ is also nonconstant,
so $T^\Delta_{\bar{a}} (V^{[d]})$ has an infinite rank subspace over $\QQQ\langle \bar{\a,} \bar{a}\rangle$, a contradiction to our assumption. 
\end{proof}

\begin{remark} 
Using Lemma \ref{morseq} together with Proposition \ref{linnonlin} gives a strategy for establishing the strong minimality of nonlinear differential equations with generic coefficients, but \emph{only if one can verify the hypothesis of Proposition \ref{linnonlin}}. A priori, this looks quite hard since it would require the analysis of systems of linear differential equations in $n$ variables for all $n \in \m N.$ This may be possible via a clever inductive argument for specially selected classes of equations, but Theorem \ref{boundMorley} gives a bound for the number of variables we need to consider. 
\end{remark} 

\begin{thm} \label{bdlinnonlin} Let $f(x)$ be a non-linear order $h\geq 1$ differential polynomial with coefficients $\bar{\a}$, let $\a_0$ be differentially transcendental over $\QQQ\langle\bar{\a}\rangle$, and let $V_0$ be the differential variety defined by $f(x)=\a_0$. Let $p$ be the generic type of $V_0$. Suppose that for some $d \geq \nmdeg{p}+1$, the following property holds: for any $\bar{a} = (a_1, \ldots , a_{d} )$ realizing the generic type of $V_0$ with $a_1, \ldots, a_d$ algebraically independent over $\QQQ\langle \bar{\a}\rangle$, the differential tangent space $T^\D_{\bar{a}}\left( V^{[d]} \right)$ has no definable proper infinite rank $\CC$-vector subspace over $\QQQ\langle \bar{\a}, \bar{a}\rangle$.
Then $V_0$ is strongly minimal.
\end{thm}

\begin{proof} 
By Proposition \ref{mindefn}, there are no subvarieties of $V_0$ defined over the differential field generated by the coefficients of $f$. So, we need only consider forking extensions of the generic type of $V_0$. By definition of nonminimality degree, if there is an infinite proper differential subvariety of $V_0$, then it is defined over (the algebraic closure of) a Morley sequence of length at most $\nmdeg(p)$. 
Thus, there is a proper subvariety of $W_0 \subset V_0^{[d]}$ which surjects onto the first $d-1$ coordinates such that the fiber over a generic point in the first $d-1$ coordinates is a forking extension of the generic type of $V_0$. But then by the argument of Proposition \ref{linnonlin}, there is a definable proper infinite rank subspace of $T^\D_{\bar{a}}\left( V^{[d]} \right)$ over $\QQQ\langle \bar{\a}, \bar{a}\rangle$.
\end{proof}




\section{Strong minimality of generic equations} \label{genstrongmin}
\subsection{A first example}
\begin{theorem}\label{thm:firstexample}
Let $X_\a$ be the differential variety given by 
\begin{equation} \label{pfcon}
x'' + \sum _{i=1} ^n \alpha _i x^i =\alpha
\end{equation}
for some $n \geq 8$, where $(\alpha, \alpha_0,\ldots , \alpha _n)$ is a tuple of independent differential transcendentals over $\m Q.$ Then $X_\a$ is strongly minimal.
\end{theorem}
\begin{proof}
By Theorem \ref{boundMorley}, we know that the degree of nonminimality of the generic type of $X_\a$ is at most 3, so it suffices to we must verify that the hypotheses of Theorem \ref{bdlinnonlin} for $d=4$.
Thus, we to show that given any solutions $a_1, a_2, a_3 ,a_4$ to equation \ref{pfcon} that are algebraically independent over $\m Q \langle \alpha , \alpha_1 , \ldots , \alpha_n \rangle $,\footnote{That is, $a_1, a_2, a_3, a_4$ satisfy no polynomial relation over $\m Q \langle \alpha , \alpha_1 , \ldots , \alpha_n \rangle $.} we cannot have that the transcendence degree of $\m Q \langle a_1, \alpha,  \alpha_1, \ldots ,  \alpha_n, a_2, \ldots , a_4 \rangle $ over  $\m Q \langle \alpha,  \alpha_1, \ldots ,  \alpha_n, a_2, \ldots , a_4 \rangle $ is one. 

Observe that the differential tangent space $T_{\bar{a}}^\Delta \left(X^{[4]}\right)$ after eliminating $y$ is given by the system:
$$
\left\{
\begin{array}{lcl}
u_0'' + \left( \sum _{i=1}^n i a_1^{i-1} \alpha_i \right) u_0  & = &  v_0'' + \left( \sum _{i=1}^n i a_2^{i-1} \alpha_i \right) v_0\\
u_0'' + \left( \sum _{i=1}^n i a_1^{i-1} \alpha_i \right) u_0  & = &  w_0'' + \left( \sum _{i=1}^n i a_3^{i-1} \alpha_i \right) w_0\\
u_0'' + \left( \sum _{i=1}^n i a_1^{i-1} \alpha_i \right) u_0  & = &  z_0'' + \left( \sum _{i=1}^n i a_4^{i-1} \alpha_i \right) z_0\\
\end{array}
\right.
$$
For $j=1, \ldots , 4$, we let $\beta _j=  \sum _{i=0}^n i a_j^{i-1} \alpha_i$. We argue that $\beta_1, \ldots , \beta_4$ are independent differential transcendentals. Note that 
$$\left( \begin{matrix}
n a_1^{n-1} & (n-1)a_1^{n-2} & \ldots & 2a_1 & 1\\
n a_2^{n-1} & (n-1)a_2^{n-2} & \ldots & 2a_2 & 1\\
n a_3^{n-1} & (n-1)a_3^{n-2} & \ldots & 2a_3 & 1\\
n a_4^{n-1} & (n-1)a_4^{n-2} & \ldots & 2a_4 & 1\\
\end{matrix}\right) \left( \begin{matrix}
\alpha_n \\
\alpha_{n-1}\\
\vdots \\
\alpha_1
\end{matrix}\right) = \left( \begin{matrix} \beta_1 \\ \beta_2 \\ \beta_3 \\ \beta_4  \end{matrix} \right)$$

We claim that any four columns of the matrix of $a_i$'s are linearly independent. To see this, note that if not then the vanishing of the corresponding determinant shows that there is a nontrivial polynomial relation which holds of $a_1, \ldots , a_4.$ 
This contradicts the fact that $a_1, \ldots a_4$ are algebraically independent over $\QQQ$. By the independence of $\alpha_1, \ldots , \alpha_n$, and since $n\geq 8$, there are at least four of the $\alpha_i$ which are independent differential transcendentals over the other $\alpha_i$ and $a_1, \ldots , a_4.$ Without loss of generality, assume $\alpha_1, \ldots , \alpha_4$ are independent differential transcendentals over $\m Q \langle \alpha_5, \ldots , \alpha _n , a_1, \ldots , a_4 \rangle $. Then since the last four columns of the above matrix of $a_i$ are linearly independent, it follows that $\alpha_1, \ldots , \alpha_4$ are  interalgebraic with $\beta_1, \ldots , \beta_4 $ over $\m Q \langle a_1, \ldots , a_4, \alpha_5, \ldots , \alpha_n \rangle.$ It follows that $\beta_1, \ldots , \beta_4$ are independent differential transcendentals over $\m Q \langle a_1, \ldots , a_4, \alpha_5, \ldots , \alpha_n \rangle$.

\begin{lemma}
\label{prop:example}
A linear system of the form 
\begin{equation}\label{exorig}
\left\{
\begin{array}{lcl}
u_0'' + \beta_1 u_0  & = &  v_0'' + \beta_2 v_0\\
u_0'' + \beta_1 u_0  & = &  w_0'' + \beta_3 w_0\\
u_0'' + \beta_1 u_0  & = &  z_0'' + \beta_4 z_0\\
\end{array}
\right.
\end{equation}
with $\beta_1, \ldots , \beta_4$ independent differential transcendentals has no proper infinite rank $\CC$-vector subspaces definable over $\UU$.
\end{lemma}

\begin{proof}
We will prove that this system has no infinite rank subvarieties by proving that the solution set is in definable bijection with $\AAA^1(\UU)$. This is constructed by composing a series of linear substitutions.

First, we substitute $u_1$ for $u_0$ where $u_0=u_1+v_0$.
This reduces the order of $v_0$ in the top equation, resulting in the system
$$
\left\{
\begin{array}{lcl}
u_1'' + \beta_1 u_1  & = &   (\beta_2-\beta_1) v_0\\
u_1'' + \beta_1 u_1 +v_0'' +\beta_1 v_0 & = &  w_0'' + \beta_3 w_0\\
u_1'' + \beta_1 u_1 +v_0'' +\beta_1 v_0  & = &  z_0'' + \beta_4 z_0\\
\end{array}
\right.
$$
To reduce the order $v_0$ in the lower equations we substitute $w_1,z_1$ for $w_0,z_0$ where $w_0=w_1+v_0,z_0=z_1+v_0$.
Then we have
$$
\left\{
\begin{array}{lcl}
u_1'' + \beta_1 u_1  & = &   (\beta_2-\beta_1) v_0\\
u_1'' + \beta_1 u_1 +(\beta_1-\b_3) v_0 & = &  w_1'' + \beta_3 w_1\\
u_1'' + \beta_1 u_1  +(\beta_1-\b_4) v_0  & = &  z_1'' + \beta_4 z_1\\
\end{array}
\right.
$$
Solving the top equation for $v_0$ in terms of $u_1$ and plugging this in for $v_0$ allows us to eliminate $v_0$ from lower equations, resulting in the system
$$
\left\{
\begin{array}{lcl}
A_{2,0}u_1'' + A_{0,0} u_1 & = &  w_1'' + \beta_3 w_1\\
C_{2,0}u_1'' + C_{0,0} u_1 & = &  z_1'' + \beta_4 z_1
\end{array}
\right.
$$
where (after some simplification)
\begin{align*}
A_{2,0}&:=\frac{\b_2-\b_3}{\b_2-\b_1}, \quad A_{0,0}:=\b_1 A_{2,0}\\
C_{2,0}&:=\frac{\b_2-\b_4}{\b_2-\b_1}, \quad C_{0,0}:=\b_1 C_{2,0}.
\end{align*}
We again reduce the order of the variable in the top equation by substituting $u_2$ for $u_1$ defined by $u_1=u_2+\frac{1}{A_{2,0}}w_1$ resulting in the system
$$
\left\{
\begin{array}{lcl}
A_{2,0}u_2'' + A_{0,0} u_2 & = &  B_{1,1}w_1'+B_{0,1} w_1\\
C_{2,0}u_2'' + C_{0,0} u_2 + D_{2,1} w_1''+ D_{1,1}w_1'+D_{0,1}w_1& = &  z_1'' + \beta_4 z_1
\end{array}
\right.
$$
where
$$
\begin{array}{lll}
 & B_{1,1}:=-2 A_{2,0}\left(\frac{1}{A_{2,0}}\right)', & \ B_{0,1}:=\b_3 - A_{2,0}\left(\frac{1}{A_{2,0}}\right)'' -\frac{A_{0,0}}{A_{2,0}}\\
D_{2,1}:=\frac{C_{2,0}}{A_{2,0}}, & D_{1,1}:=2C_{2,0}\left(\frac{1}{A_{2,0}}\right)', & D_{0,1}:=C_{2,0}\left(\frac{1}{A_{2,0}}\right)''+\frac{C_{0,0}}{A_{2,0}}.
\end{array}
$$
We next reduce the order of $w_1$ in lower equations with the substitution $z_2$ for $z_1$ defined by $z_1=z_2+D_{2,1}w_1$.
Now we have the system 
$$
\left\{
\begin{array}{lcl}
A_{2,0}u_2'' + A_{0,0} u_2 & = &  B_{1,1}w_1'+B_{0,1} w_1\\
C_{2,0}u_2'' + C_{0,0} u_2 + E_{1,1}w_1'+E_{0,1}w_1& = &  z_2'' + \beta_4 z_2
\end{array}
\right.
$$
where
$$
E_{1,1}:=D_{1,1}-D_{2,1}', \quad E_{0,1}:=D_{0,1} - D_{2,1}'' -\b_4 D_{2,1}.
$$
Next we reduce the order of $u_2$ in the top equation by substituting $w_2$ for $w_1$ with $w_1=w_2 + \frac{A_{2,0}}{B_{1,1}}u_2'$ resulting in the system
$$
\left\{
\begin{array}{lcl}
A_{1,1}u_2' + A_{0,1} u_2 & = &  B_{1,1}w_2'+B_{0,1} w_2\\
C_{2,1}u_2'' +C_{1,1}u_2'+ C_{0,1} u_2 + E_{1,1}w_2'+E_{0,1}w_2& = &  z_2'' + \beta_4 z_2\\
\end{array}
\right.
$$
where
\begin{align*}
A_{1,1}&:=-B_{1,1}\left(\frac{A_{2,0}}{B_{1,1}}\right)'-B_{0,1}\frac{A_{2,0}}{B_{1,1}}, \quad A_{0,1}:= A_{0,0}\\
C_{2,1}&:=C_{2,0}+E_{1,1}\frac{A_{2,0}}{B_{1,1}}, \ C_{1,1}:=E_{1,1}\left(\frac{A_{2,0}}{B_{1,1}}\right)'+E_{0,1}\frac{A_{2,0}}{B_{1,1}}, \quad C_{0,1}:=C_{0,0}.
\end{align*}
The reduction in order of the top equation continues with the replacement of $u_2$ with $u_3$ given by $u_2=u_3+\frac{B_{1,1}}{A_{1,1}}w_2$.
This results in the system
$$
\left\{
\begin{array}{lcl}
A_{1,1}u_3' + A_{0,1} u_3 & = &  B_{0,2} w_2\\
C_{2,1}u_3'' +C_{1,1}u_3'+ C_{0,1} u_3 + D_{2,2}w_2''+ D_{1,2}w_2'+D_{0,2}w_2& = &  z_2'' + \beta_4 z_2
\end{array}
\right.
$$
where
\begin{align*}
B_{0,2}&:=B_{0,1}- A_{1,1}\left(\frac{B_{1,1}}{A_{1,1}}\right)'-A_{0,1}\frac{B_{1,1}}{A_{1,1}} \\
D_{2,2}&:= C_{2,1}\frac{B_{1,1}}{A_{1,1}}, \quad D_{1,2}:=E_{1,1}+2 C_{2,1}\left(\frac{B_{1,1}}{A_{1,1}}\right)'+C_{1,1}\frac{B_{1,1}}{A_{1,1}},\\
D_{0,2}&:=E_{0,1}+ C_{2,1}\left(\frac{B_{1,1}}{A_{1,1}}\right)'' +C_{1,1}\left(\frac{B_{1,1}}{A_{1,1}}\right)' +C_{0,1}\frac{B_{1,1}}{A_{1,1}}.
\end{align*}
Next we replace $z_2$ with $z_3$ given by $z_2=z_3 + D_{2,2}w_2$ to arrive at
$$
\left\{
\begin{array}{lcl}
A_{1,1}u_3' + A_{0,1} u_3 & = &  B_{0,2} w_2\\
C_{2,1}u_3'' +C_{1,1}u_3'+ C_{0,1} u_3 +  E_{1,2}w_2'+E_{0,2}w_2& = &  z_3'' + \beta_4 z_3
\end{array}
\right.
$$
where
$$
E_{1,2}:=D_{1,2}-D_{2,2}', \quad E_{0,2}:=D_{0,2} - D_{2,2}'' -\b_4 D_{2,2}.
$$
Now we can solve the top equation for $w_2$ in terms of $u_3$ and plug the resulting expression into the lower equations:
$$
C_{2,2}u_3'' +C_{1,2}u_3'+ C_{0,2} u_3  =  z_3'' + \beta_4 z_3
$$
where
$$
\begin{array}{l}
C_{2,2}:= C_{2,1}+E_{1,2}\frac{A_{1,1}}{B_{0,2}}, \\
C_{1,2}:= C_{1,1}+E_{1,2}\left(\frac{A_{1,1}}{B_{0,2}}\right)' + E_{0,2}\frac{A_{1,1}}{B_{0,2}} + E_{1,2}\frac{A_{0,2}}{B_{0,2}},\\
C_{0,2}:=C_{0,1}+E_{1,2}\left(\frac{A_{0,1}}{B_{0,2}}\right)'+E_{0,2}\frac{A_{0,1}}{B_{0,2}}.
\end{array}
$$
Next we perform analagous substitutions to eliminate $z_3$ from the top equation, beginning with substituting $u_4$ for $u_3$ defined by $u_3=u_4+\frac{1}{C_{2,2}}z_3$, so we have
$$
\left\{
\begin{array}{lcl}
C_{2,2}u_4'' +C_{1,2}u_4'+ C_{0,2} u_4 & = &  F_{1,1}z_3' + F_{0,1} z_3\\
\end{array}
\right.
$$
where
$$
F_{1,1}:=-2C_{2,2}\left(\frac{1}{C_{2,2}}\right)'-\frac{C_{1,2}}{C_{2,2}}, \quad F_{0,1}:=\b_4-C_{2,2}\left(\frac{1}{C_{2,2}}\right)''-C_{1,2}\left(\frac{1}{C_{2,2}}\right)'-\frac{C_{0,2}}{C_{2,2}}.
$$
Next, substitute $z_4$ for $z_3$ where $z_3=z_4+\frac{C_{2,2}}{F_{1,1}}u_4'$.
This will result in the equation 
$$
C_{1,3}u_4'+ C_{0,3} u_4 =  F_{1,1}z_4' + F_{0,1} z_4
$$
where
$$
C_{1,3}:=C_{1,2}-F_{1,1}\left(\frac{C_{2,2}}{F_{1,1}}\right)'+F_{0,1}\frac{C_{2,2}}{F_{1,1}}, \quad C_{0,3}:=C_{0,2}.
$$
Replace $u_4$ with $u_5$ defined by $u_4=u_5+\frac{F_{1,1}}{C_{1,3}}z_4$, giving us the equation
\begin{equation}\label{exfinal}
C_{1,3}u_5'+ C_{0,3} u_5 =  F_{0,2} z_4
\end{equation}
where
$$
F_{0,2}:=F_{0,1}-C_{1,3}\left(\frac{F_{1,1}}{C_{1,3}}\right)'+F_{0,1}\frac{F_{1,1}}{C_{1,3}}.
$$

Any solution to Equation \ref{exfinal} is determined by the value of $u_5$, so the solution set is in definable bijection with $\AAA^1(\UU)$.
Each of these linear substitutions gives rise to a definable bijection between systems so long as the substitutions are well-defined, i.e., that the denominators of the coefficients are all non-zero. For this procedure to give a definable bijection from the original system \ref{exorig} to $\AAA^1(\UU)$, we must verify that the following expressions are not zero:
$$\b_2-\b_1, A_{2,0}, B_{1,1}, A_{1,1}, B_{0,2}, C_{2,2}, F_{1,1}, C_{1,3}, \text{ and }F_{0,2}.$$
The expression $\b_2-\b_1$ is non-zero because $\b_1,\b_2,\b_3,\b_4$ are all distinct.
Each of these coefficients can be considered as a differential rational function in terms of $\b_1,\b_2, \b_3,\b_4$, and so they can be analyzed according to a ranking on $\bar{\b}$. We will show that these coefficients are nonzero by showing that the initials of each are nonzero in some elimination ranking. It then follows that the coefficients themselves are non-zero because $\b_1, \b_2, \b_3, \b_4$ are independent differential transcendentals.

Consider the terms of these expressions ordered by some elimination ranking on $\bar{\b}$ with $\b_3$ ranked highest.
The leading term in this ranking of each expression can be calculated using the definitions of previous coefficients. The following table shows that these leading terms are non-zero:
$$
\begin{array}{c|l}
A_{2,0} & \frac{-1}{\b_2-\b_1}\b_3 \\
B_{1,1} & \frac{-2}{\b_2-\b_3}\b_3' \\
A_{1,1} & \frac{-2}{(\b_2-\b_1)B_{1,1}}\b_3'' \\
B_{0,2} & \frac{-2}{(\b_2-\b_1)A_{1,1}}\b_3^{(3)}
\end{array}
$$
We turn our attention to an elimination ranking with $\b_4$ ranked highest to prove that the remaining coefficients are nonzero. The following table shows that the leading terms of these coefficients are also non-zero:
$$
\begin{array}{c|l}
C_{2,2} & \frac{-3}{(\b_2-\b_1)B_{0,2}}\b_4''\\
F_{1,1} & \frac{-4}{(\b_2-\b_1)B_{0,2}C_{2,2}}\b_4^{(3)}\\
C_{1,3} & \frac{-7}{(\b_2-\b_1)B_{0,2}F_{1,1}}\b_4^{(4)} \\ F_{0,2} & \frac{-7}{(\b_2-\b_1)B_{0,2}C_{1,3}}\b_4^{(5)}
\end{array}
$$

We have shown that each substitution is well-defined, and therefore we have constructed a definable bijection between system \ref{exorig} and $\AAA^1(\UU)$. Since $\AAA^1(\UU)$ has no infinite rank subspaces, neither does our original system, completing the proof of the proposition.
\end{proof}
We now finish the proof of Theorem \ref{thm:firstexample}.
Since the differential tangent space $T_{\bar{a}}^\Delta \left( X^{[4]}\right)$ satisfies the conditions of Lemma \ref{prop:example}, it has no infinite rank subspaces over $\QQQ\langle \bar{\a},\bar{a}\rangle$. Therefore, $X_\a$ is strongly minimal by Theorem \ref{bdlinnonlin}.
\end{proof}

\subsection{Generic higher order equations} 
The technique used in the previous example can be applied to more general classes of equations. In this section, we use analogous techniques to show that generic equations with high enough degree have differential tangent spaces cut out by linear equations with generic coefficients and that these tangent spaces have no proper infinite rank $\CC$-vector subspace definable over $\UU$. 

Let $$f(x)=\alpha + \sum_{i=1}^d \alpha_{0,i} x^i + \sum_{j\in M_1} \a_{1,j} m_j\left(x,x'\right) + \cdots + \sum _{k\in M_h} \a_{h,k} m_k \left(x,x',\ldots , x^{(h)}\right)$$ where $M_n$ indexes the set of all order $n$ monomials of degree at most $d$ and the entire collection of coefficients $\alpha, \alpha_{i,j}$ are independent differential transcendentals over $\m Q$. Let $V_\a$ be the zero set of $f(x)$ and let $m$ be the degree of nonmininality of $f$ plus one. Following the notation of Section \ref{Lintrans}, we let $V^{[m]}$ be the following system of equations in $x_1, \ldots , x_m , y$: 
\[
\left\{
\begin{array}{lcl}
\sum_{i=0}^h \sum_{j\in M_i}\alpha_{i,j} m_j\left(x_1, \ldots, x_1^{(i)}\right)  & = & y \\
 \sum_{i=0}^h \sum_{j\in M_i}\alpha_{i,j} m_j\left(x_2, \ldots, x_2^{(i)}\right)   & = & y    \\
  & \vdots &  \\
 \sum_{i=0}^h \sum_{j\in M_i}\alpha_{i,j} m_j\left(x_m, \ldots, x_m^{(i)}\right)   & = & y 
\end{array}
\right.
\]

Let $\bar a = (a_1, \ldots , a_m)$ be an indiscernible sequence in $V_\a$ such that $a_m$ forks over $a_1, \ldots , a_{m-1}$ and $tp(a_m/ \m Q \langle \alpha, \alpha_{i,j} , a_1, \ldots , a_{m-1} \rangle _{i=0, \ldots, h, j\in M_i } $ has rank between $1$ and $h-1$. That is, $a_m $ satisfies a differential equation of order at least $1$ but no more than $h-1$ over $\m Q \langle \alpha, \alpha_{i,j} , a_1, \ldots , a_{m-1} \rangle _{i=0, \ldots, h, j\in M_i } $. Crucially for this proof, we note that $a_1, \ldots , a_m$ are \emph{algebraically} independent over $\m Q \langle \alpha, \alpha_{i,j} \rangle $. 

Let $T^\D_{(\bar{a}, \alpha)}\left( V^{[m]} \right)$ denote the differential tangent space of $V^{[m]}$ over $(\bar a, \alpha)$. Then $T^\D_{(\bar{a}, \alpha)}\left( V^{[m]} \right)$ is given by 

\[
\left\{
\begin{array}{lcl}
\sum_{i=0}^h \b_{i,1} z_1^{(i)}  & = & y \\
\sum_{i=0}^h \b_{i,2} z_2^{(i)}  & = & y \\
  & \vdots &  \\
\sum_{i=0}^h \b_{i,m} z_m^{(i)}  & = & y
\end{array}
\right.
\] 
where $\b_{i,j}=\frac{\partial f}{\partial x^{(i)}}(a_j)$ as used in previous sections.
\begin{lemma}\label{lem:higherordergeneric} Let $V$ be a differential variety cut out by a generic differential polynomial of order $h$, degree $d$ and nonminimality degree no more than $m-1$. If $\bar a = (a_1, \ldots , a_m)$ are realizations of the generic type of $V_0$ such that $a_1,\ldots, a_m$ are algebraically independent over $\QQQ\langle \bar{\a}\rangle$, then the variety $T^\D_{(\bar{a}, \alpha)}\left( V^{[m]} \right)$ has coefficients which are independent differential transcendentals over $\m Q$ whenever $d \geq 2m.$ 
\end{lemma}
\begin{proof}

We proceed by induction on the order, beginning with $\b_{0,1}, \b_{0,2}, \ldots, \b_{0,m}$. Since $d \geq 2m, $ there are $j_1, \ldots , j_m$ such that $\alpha_{0,j_1}, \ldots , \alpha_{0,j_m}$ are independent differential transcendentals over $\m Q \left\langle \mathcal{A}_0 a_1 \cdots  a_{m} \right\rangle $ where $\mathcal{A}_0=\{\a_{i,j}:0\leq i \leq h, j\in M_i\}\setminus \{\a_{0,j_1}, \ldots, \a_{0,j_m}\}$. Note that 

$$\left( \begin{matrix}
j_1 a_{1}^{j_1-1} & j_2 a_{1}^{j_2-1} & \ldots & j_m a_{1} ^{j_m -1}\\
j_1 a_{2}^{j_1-1} & j_2 a_{2}^{j_2-1} & \ldots & j_m a_{2} ^{j_m -1}\\
\vdots & & & \vdots \\
j_1 a_{m}^{j_1-1} & j_2 a_{m}^{j_2-1} & \ldots & j_m a_{m} ^{j_m -1}
\end{matrix}\right) \left( \begin{matrix}
\alpha_{0,j_1} \\
\alpha_{0,j_2}\\
\vdots \\
\alpha_{0,j_m}
\end{matrix}\right) = \left( \begin{matrix} \b_{0,1} \\ \b_{0,2} \\ \vdots \\ \b_{0,m} \end{matrix} \right) + \left( \begin{matrix} l_1 \\ l_2 \\ \vdots \\ l_m \end{matrix} \right),$$ 

where $l_i \in \m Q \left\langle \mathcal{A}_0 a_1 \cdots  a_{m} \right\rangle  $.

The above matrix is invertible, as the vanishing of its determinant imposes a nontrivial algebraic relation among $a_1, \ldots , a_m$, which are, by assumption, algebraically independent. It follows that $\b_{0,1}, \ldots, \b_{0,m}$ are interdefinable with $\alpha_{0,j_1}, \ldots , \alpha_{0,j_m}$ over $\m Q \left\langle \mathcal{A}_0 a_1 \cdots  a_{m} \right\rangle .$ Thus, $\b_{0,1}, \ldots, \b_{0,m}$ are independent and differentially transcendental over $\m Q \left\langle \mathcal{A}_0 a_1 \cdots  a_{m} \right\rangle .$ Note that the coefficients $\{\b_{i,j}:1\leq i \leq h, 1\leq j \leq m\} $ are contained in the field $\m Q \left\langle \mathcal{A}_0 a_1 \cdots  a_{m} \right\rangle ,$ so we've shown that $\b_{0,1}, \ldots, \b_{0,m}$ are independent differential transcendentals over $\m Q \langle \b_{i,j}:1\leq i \leq h, 1\leq j \leq m\rangle.$ 

Suppose we have already shown that $\b_{n,1}, \ldots, \b_{n,m}$ are independent differential transcendentals over $\m Q \langle \b_{i,j}:n+1\leq i \leq h, 1\leq j \leq m\rangle$ for some $n<h$.
Let $M^*_{n+1}$ index the collection of order $n+1$ monomials of order no more than $d$ excluding the monomials of the form $\{x^{(n+1)}x^r:0\leq r \leq d-1\}$.
Assume the order $n+1$ terms in $f(x)$ are ordered so that
$$\sum _{k\in M_{n+1}} \a_{n+1,k} m_k (x,x',\ldots , x^{(n+1)}) = \sum _{k=0}^{d-1} \alpha_{n+1,k} x^{(n+1)} x^k + \sum_{j \in M^*_{n+1}} \alpha_{n+1,j} m_j \left(x,x', \ldots, x^{(n+1)}\right). $$
Since $d \geq 2m, $ there are $k_1, \ldots , k_m <d$ such that $\alpha_{n+1,k_1}, \ldots , \alpha_{n+1,k_m}$ are independent and differentially transcendental over $\m Q \langle\mathcal{A}_{n+1} a_1 \cdots a_m \rangle$ where $\mathcal{A}_{n+1}=\{\alpha_{i,j}: n+1\leq i\leq h,j\in M_{i}\}\setminus \{\alpha_{n+1,k_1}, \ldots , \alpha_{n+1,k_m}\}$. Note that 

$$\left( \begin{matrix}
a_{1}^{k_1} & a_{1}^{k_2} & \ldots & a_{1} ^{k_m }\\
a_{2}^{k_1} & a_{2}^{k_2} & \ldots & a_{2} ^{k_m }\\
\vdots & & & \vdots \\
a_{m}^{k_1} & a_{m}^{k_2} & \ldots & a_{m} ^{k_m }
\end{matrix}\right) \left( \begin{matrix}
\alpha_{n+1,k_1} \\
\alpha_{n+1,k_2}\\
\vdots \\
\alpha_{n+1,k_m}
\end{matrix}\right) = \left( \begin{matrix} \b_{n+1,1} \\ \b_{n+1,2} \\ \vdots \\ \b_{n+1,m} \end{matrix} \right) + \left( \begin{matrix} r_1 \\ r_2 \\ \vdots \\ r_m \end{matrix} \right),$$ 
where $r_i \in \m Q \langle\mathcal{A}_{n+1} a_1 \cdots a_m \rangle$. The above matrix is invertible, since we are assuming that $a_1, \ldots , a_m$ are algebraically independent. Thus, $\alpha_{n+1,k_1}, \ldots , \alpha_{n+1,k_m}$ are interdefinable with $\beta_{n+1,1}, \ldots , \beta_{n+1,m}$ over  $\m Q \langle\mathcal{A}_{n+1} a_1 \cdots a_m \rangle$. Since $\{\beta_{i,j}:n+1<i\leq h, 1\leq j \leq m\}$ are contained in the field $\m Q \langle\mathcal{A}_{n+1} a_1 \cdots a_m \rangle$, it follows that $\beta_{n+1,1}, \ldots , \beta_{n+1,m}$ are independent and differentially transcendental over $\{\beta_{i,j}:n+1<i\leq h, 1\leq j \leq m\}$.

Putting together the above analysis, we have proved that the collection of coefficients $\{\beta_{i,j}:0\leq i \leq h, 1\leq j \leq m\}$ are independent and differentially transcendental over $\m Q \langle  a_1, \ldots, a_m \rangle$ and thus over $\m Q.$ 
\end{proof}


Eliminating the variable $y$ from $T^\D_{(\bar{a}, \alpha)}\left( V^{[m]} \right)$ results in a system of $m-1$ linear equations in $m$ variables as described in the following theorem.

\begin{thm}\label{thm:defnbijection} Consider a system of $m-1$ linear differential equations of order $h$ in variables $(z_1,\ldots, z_m)$ of the form: \begin{equation}\label{eq:initialsys}
\left\{
\begin{array}{lcl}
\sum_{i=0}^h \b_{i,1} z_1^{(i)}  & = & \sum_{i=0}^h \b_{i,2} z_2^{(i)} \\
\sum_{i=0}^h \b_{i,1} z_1^{(i)}  & = & \sum_{i=0}^h \b_{i,3} z_3^{(i)} \\
  & \vdots &  \\
\sum_{i=0}^h \b_{i,1} z_1^{(i)}  & = & \sum_{i=0}^h \b_{i,m} z_m^{(i)}.
\end{array}
\right.
\end{equation}
where the entire set of coefficients $\{\b_{i,j}:0\leq i \leq h, 1\leq j \leq m\}$ are independent differential transcendentals. Then the solution set of such a system has no 
proper infinite rank $\CC$-vector subspaces definable over $\UU$ for $h>1$ and $m>1$. 
\end{thm} 

As in Proposition \ref{prop:example}, we show that such a system has no infinite rank subspaces by constructing a definable bijection to $\AAA^1(\UU)$.
This is accomplished through two algorithms:
Algorithm A will apply a sequence of linear substitutions which will lower the order of the top equation by one, and Algorithm B will apply Algorithm A inductively to reduce the number of equations and variables by one.

\textbf{Algorithm A}:

We begin with a system of equations of the form
\[
\left\{
\begin{array}{lcl}
\sum_{i=0}^{\ell} a_i u^{(i)} & = & \sum_{i=0}^{\ell} b_i z^{(i)} \\
\sum_{i=0}^h c_i u^{(i)} + \sum_{i=0}^{h-1} e_i z^{(i)}  & = & \sum_{i=0}^h \b_i w^{(i)} \\
\end{array}
\right.
\]
with $1\leq \ell \leq h$. We lower the order of the top equation by applying a trio of substitutions.
The first substitution replaces the variable $u$ with $\tilde{u}$ where $u=\tilde{u} + \left(\frac{b_\ell}{a_\ell}\right)z$ which reduces the order of $z$ in the top equation by one.
This results in the new system:
\[
\left\{
\begin{array}{lcl}
\sum_{i=0}^{\ell} a_i \tilde{u}^{(i)} & = & \sum_{i=0}^{\ell-1} \tilde{b}_i z^{(i)} \\
\sum_{i=0}^h c_i \tilde{u}^{(i)} + \sum_{i=0}^{h} d_i z^{(i)}  & = & \sum_{i=0}^h \b_i w^{(i)} \\
\end{array}
\right.
\]
where
\begin{align*}
d_h&:= c_h\left(\frac{b_{\ell}}{a_{\ell}}\right),\\
d_{i}&:= e_{i}+\sum_{k=0}^{h-i}\binom{h}{h-i-k}c_{h-k}\left(\frac{b_{\ell}}{a_\ell}\right)^{(h-i-k)}
\end{align*}
for $0\leq i \leq h-1$, and
$$
 \tilde{b}_{i}:= b_{i}-\sum_{k=0}^{\ell-i}\binom{\ell}{\ell-i-k}a_{\ell-k}\left( \frac{b_{\ell}}{a_{\ell}}\right)^{(\ell-i-k)}
 $$
 for $0\leq i \leq \ell-1$.
 
 The next substitution replaces $v$ with $\tilde{v}$, defined by $v=\tilde{v}+\left(\frac{d_{h}}{\b_{h}}\right)z$, reducing the order of $z$ in the lower equation by one.
This results in the system of equations
\[
\left\{
\begin{array}{lcl}
\sum_{i=0}^{\ell} a_{i} \tilde{u}^{(i)} & = & \sum_{i=0}^{\ell-1} \tilde{b}_{i} z^{(i)} \\
\sum_{i=0}^h c_{i} \tilde{u}^{(i)} + \sum_{i=0}^{h-1} \tilde{e}_{i} z^{(i)}  & = & \sum_{i=0}^h \b_{i} \tilde{v}^{(i)} \\
\end{array}
\right.
\]
where
\[
\tilde{e}_{i}:= d_{i}- \sum_{k=0}^{h-i}\binom{h}{h-i-k}\b_{h-k}\left(\frac{d_{h}}{\b_{h}}\right)^{(h-i-k)}
\]
for $0\leq i \leq h-1$.

To complete the trio, we substitute $\tilde{z}$ for $z$ defined by $z=\tilde{z}+\left(\frac{a_{\ell}}{\tilde{b}_{\ell-1}}\right)\tilde{u}'$.
Now we have the system
\[
\left\{
\begin{array}{lcl}
\sum_{i=0}^{\ell-1} \tilde{a}_{i} \tilde{u}^{(i)} & = & \sum_{i=0}^{\ell-1} \tilde{b}_{i} \tilde{z}^{(i)} \\
\sum_{i=0}^h \tilde{c}_{i} \tilde{u}^{(i)} + \sum_{i=0}^{h-1} \tilde{e}_{i} \tilde{z}^{(i)}  & = & \sum_{i=0}^h \b_{i} \tilde{v}^{(i)} \\
\end{array}
\right.
\]
where 
\begin{align*}
\tilde{c}_{i}&:= c_{i} + \sum_{k=0}^{h-i} \binom{h-1}{h-i-k} \tilde{e}_{h-1-k} \left( \frac{a_{\ell}}{\tilde{b}_{\ell-1}}\right)^{(h-i-k)}\\
\tilde{c}_{0}&:=c_{0}
\end{align*}
for $1\leq i \leq h$
and 
\begin{align*}
    \tilde{a}_{i}&:= a_{i}-\sum_{k=0}^{\ell-i}\binom{\ell-1}{\ell-i-k}\tilde{b}_{\ell-1-k}\left( \frac{a_{\ell}}{\tilde{b}_{\ell-1}}\right)^{(\ell-i-k)}\\
    \tilde{a}_{0}&:=a_{0}
\end{align*}
for $1\leq i\leq \ell-1$.

We must first prove that each of these linear substitutions is well-defined under the assumptions of Theorem \ref{thm:defnbijection}, i.e., that the coefficients $a_\ell$, $\b_h$, and $\tilde{b}_{\ell-1}$ which appear in the denominators of substitutions are all non-zero.
\begin{lemma}\label{lem:ABnew}
The tuple $(b_{i}, a_{i}:0\leq i \leq \ell )$ is interdefinable with $(b_\ell, a_\ell, \tilde{b}_{i}, \tilde{a}_{i}:0\leq i \leq \ell-1 )$ as defined in Algorithm A.
\end{lemma}
\begin{proof}
It is clear from the definition of $\tilde{b}_{i}$ that $(b_{i}:0\leq i \leq \ell-1)$ is interdefinable with $(\tilde{b}_{i}:0\leq i \leq \ell-1)$ over $b_\ell$ and $\{a_{i}:0\leq i \leq \ell\}$.
By adding the parameters themselves, we have 
$(b_{i}, a_{i}:0\leq i \leq \ell)$ is interdefinable with $(b_{\ell}, \tilde{b}_{i}:0\leq i \leq \ell-1)\cup( a_{i}:0\leq i \leq \ell)$.
It is also clear from the definition of $\tilde{a}_{i}$ that
$(a_{i}: 0\leq i \leq \ell-1)$ is interdefinable with $(\tilde{a}_{i}:0\leq i \leq \ell-1)$ over $a_{\ell}$ and $\{\tilde{b}_{i}:0\leq i \leq \ell-1\}$.
Combining these two facts, we conclude that the desired tuples are interdefinable.
\end{proof}
Assuming the initial coefficients $\{a_i,b_i:0\leq i \leq \ell\}\cup\{\b_i:0\leq i \leq h\}$ constitute an independent set of differential transcendentals, Lemma \ref{lem:ABnew} shows that  $\{\tilde{a}_i,\tilde{b}_i:0\leq i \leq \ell-1\}\cup\{\b_i:0\leq i \leq h\}\cup\{a_\ell,b_\ell\}$ are also independent differential transcendentals, and hence non-zero. It follows that the substitutions in Algorithm A are well-defined as long as the initial coefficients have this property.
\begin{lemma}\label{lem:alga}
Suppose that $(e_i:0\leq i \leq h-1)$ is differentially algebraic over $\{c_i:0\leq i \leq h\}\cup\{a_\ell, b_\ell, \tilde{a}_{\ell -1}, \tilde{b}_{\ell-1}, \}\cup\{\b_i:0 \leq i \leq h\}$. Then $(c_i:0\leq i \leq h)$ is inter-differentially algebraic with $(\tilde{c}_i:0\leq i \leq h)$ over $\{a_i,b_i:0\leq i \leq \ell\}\cup \{\b_i:0\leq i \leq \ell\}$, and $(\tilde{e}_i:0\leq i \leq h-1)$ is differentially algebraic over $\{\tilde{c}_i:0\leq i \leq h\}\cup\{a_\ell, b_\ell, \tilde{a}_{\ell -1}, \tilde{b}_{\ell-1}, \}\cup\{\b_i:0 \leq i \leq h\}$.
\end{lemma}
\begin{proof}
From the definitions, it is clear that $(\tilde{c}_i:0\leq i \leq h)$ is differentially algebraic over $(c_i:0\leq i \leq h)$. Now to prove the other direction, we combine the definitions of $\tilde{c}_i$, $\tilde{e}_i$, and $d_i$, we can write $\tilde{c}_i$ in terms of $(c_i:0\leq i \leq h)$ (as well as other terms). Doing this for $\tilde{c}_h$ (and performing some simplifications), we have
\begin{equation}\label{eq:ch}
    \tilde{c}_h = \frac{b_\ell}{\tilde{b}_{\ell-1}} c_{h-1}-\frac{ h b_{\ell}}{ \tilde{b}_{\ell-1}}c_{h}'
    +\left[ 1+ \frac{a_\ell}{\tilde{b}_{\ell-1}}  \left(h\left(\frac{b_\ell}{a_\ell}\right)'-h\b_h\left( \frac{b_\ell}{\b_h a_\ell} \right)' - \frac{\b_{h-1}b_\ell}{\b_h a_\ell} \right)\right] c_h + \frac{a_\ell}{\tilde{b}_{\ell-1}} e_{h-1}.
\end{equation}
Doing the same other $\tilde{c}$'s, we see that $c_h$ appears with order at least one, but each $c_i$ is order zero and linear over $\QQQ\langle a_\ell, \tilde{a}_{\ell-1},b_\ell, \tilde{b}_{\ell-1}\rangle$ for $i<h$. This is represented in the following matrix equation:
\begin{equation}\label{eq:matrix}
\begin{bmatrix}
M_1 & M_0 & 0 &  \cdots & 0 & 0 & 0 & 0\\
M_2 & M_1 & M_0 &  \cdots & 0 & 0 & 0 & 0\\
M_3 & M_2 & M_1 &  \cdots & 0 & 0 & 0 & 0\\
\vdots & \vdots & \vdots &  \ddots & \vdots & \vdots & \vdots & \vdots\\
M_{h-3} & M_{h-4} & M_{h-5} & \cdots & M_1 & M_0 & 0 & 0\\
M_{h-2} & M_{h-3} & M_{h-4} & \cdots & M_2 & M_1 & M_0 & 0\\
M_{h-1} & M_{h-2} & M_{h-3} &  \cdots & M_3 & M_2 & M_1 & M_0\\
0 & 0 & 0 & \cdots &0 & 0 & 0 & 1
\end{bmatrix}
\begin{bmatrix}
c_{h-1} \\ c_{h-2} \\ c_{h-3} \\ \vdots  \\ c_{3} \\ c_{2} \\ c_{1} \\ c_{0}
\end{bmatrix}
=
\begin{bmatrix}
\tilde{c}_{h-1} \\ \tilde{c}_{h-2} \\ \tilde{c}_{h-3} \\ \vdots  \\ \tilde{c}_{3} \\ \tilde{c}_{2} \\ \tilde{c}_{1} \\ \tilde{c}_{0}
\end{bmatrix}
-
\begin{bmatrix}
L_{h-1} \\ L_{h-2} \\ L_{h-3} \\ \vdots \\ L_{3} \\ L_{2} \\ L_{1} \\ 0
\end{bmatrix}
\end{equation}
where $L_i\in \QQQ\left\langle \{c_{h},a_{\ell},b_{\ell},\tilde{a}_{\ell-1}, \tilde{b}_{\ell-1}\}\cup\{\b_i:0 \leq i \leq h\}\cup\{e_i:0\leq i \leq h-1\}\right\rangle$.
The matrix entries can be computed as follows:
\begin{align*}
M_0 &= \frac{b_{\ell}}{\tilde{b}_{\ell-1}}\\
M_1 &= h \left(\frac{b_{\ell}}{a_{\ell}}\right)' \left(\frac{a_{\ell}}{\tilde{b}_{\ell-1}}\right) + (h-1)\left(\frac{b_{\ell}}{a_{\ell}}\right)\left(\frac{a_{\ell}}{\tilde{b}_{\ell-1}}\right)'+1\\
M_i &= \sum_{k=0}^i\binom{h-1}{k}\binom{h}{i-k}\left(\frac{b_{\ell}}{a_{\ell}}\right)^{(i-k)} \left(\frac{a_{\ell}}{\tilde{b}_{\ell-1}}\right)^{(k)}
\end{align*}
for $i>1$.
Using Gaussian elimination, we can make this matrix lower triangular:
\[
M^*=\begin{bmatrix}
M_{1,h-1} & 0 & 0 & \cdots & 0 & 0 & 0 & 0\\
M_{2,h-2} & M_{1,h-2} & 0 & \cdots & 0 & 0 & 0 & 0\\
M_{3,h-3} & M_{2,h-3} & M_{1,h-3} & \cdots & 0 & 0 & 0 & 0\\
\vdots & \vdots & \vdots & \ddots & \vdots & \vdots & \vdots & \vdots\\
M_{h-3,3} & M_{h-4,3} & M_{h-5,3} & \cdots & M_{1,3} & 0 & 0 & 0\\
M_{h-2,2} & M_{h-3,2} & M_{h-4,2} & \cdots & M_{2,2} & M_{1,2} & 0 & 0\\
M_{h-1,1} & M_{h-2,1} & M_{h-3,1} & \cdots & M_{3,1} & M_{2,1} & M_{1,1} & 0\\
0 & 0 & 0 & \cdots &0 & 0 & 0 & 1
\end{bmatrix}
\]
where the new matrix entries can be computed recursively for $2\leq j\leq h-1$ and $1\leq i \leq h-j$
$$
    M_{i,1}:=M_{i}, \quad
    M_{i,j}:=M_{i}-\left(\frac{M_0}{M_{1,j-1}}\right) M_{i+1,j-1}.
$$
\begin{claim} For all $1\leq j\leq h-1$ and $1\leq i \leq h-j$, $M_{i,j}\neq 0$, and therefore, the determinant of $M^*$ is non-zero.
\end{claim}
\begin{proof}
We prove this by showing that $M_{i,j}$ is order $i+j-1$ in $a_{\ell}$ using induction on $j$.
It follows from Claim \ref{lem:ABnew} that $M_{i,1}$ is order $i$ in $a_{\ell}$ for each $i$.
Now suppose the claim holds for $j$.
Since $M_{i,j+1}=M_{i}-\left(\frac{M_0}{M_{1,j}}\right) M_{i+1,j}$, we can see that $M_{i,j+1}$ is order $i+j$ since $M_{i+1,j}$ is order $i+j$ by assumption, and all other terms in the definition have strictly smaller order. Thus the determinant is non-zero. 
\end{proof}

Since $M^*$ is non-singular, it follows also that the original matrix in Equation \ref{eq:matrix} is non-singular. Therefore, $(c_{i}:0\leq i \leq h-1)$ and $(\tilde{c}_{i}:0\leq i \leq h-1)$ are interdefinable over $\{c_{h},a_{\ell},b_{\ell},\tilde{a}_{\ell-1}, \tilde{b}_{\ell-1}\}\cup \{\b_i:0 \leq i \leq h\}\cup\{e_i:0\leq i \leq h-1\}$.

Turning our attention once more to Equation \ref{eq:ch}, we use this interdefinability to substitute $c_{h-1}$ for an expression involving $c_h$ and $(\tilde{c}_i:0\leq i \leq h)$.
The resulting equation establishes a differential relation between $c_h$ and $(\tilde{c}_i:0\leq i \leq h)$.
It follows that $(c_i:0\leq i \leq h)$ is differentially algebraic over $\tilde{c}_i:0\leq i \leq h)$ along with $\{a_\ell, b_\ell, \tilde{a}_{\ell-1}, \tilde{b}_{\ell-1}\}\cup \{\b_i:0 \leq i \leq h\}\cup\{e_i:0\leq i \leq h-1\}$.
By assumption, $(e_i:0\leq i \leq h-1)$ is identically 0 or differentially algebraic over $\{c_i:0\leq i \leq h\}\cup\{a_\ell, b_\ell, \tilde{a}_{\ell -1}, \tilde{b}_{\ell-1}, \}\cup\{\b_i:0 \leq i \leq h\}$. In either case, it follows that $(c_i:0\leq i \leq h)$ and $\tilde{c}_i:0\leq i \leq h)$ are inter-differentially algebraic over $\{a_\ell, b_\ell, \tilde{a}_{\ell-1}, \tilde{b}_{\ell-1}\}\cup \{\b_i:0 \leq i \leq h\}$.

So that this lemma can be applied inductively, we show that $(\tilde{e}_i:0\leq i \leq h-1)$ is differentially algebraic over $\{\tilde{c}_i:0\leq i \leq h\}\cup\{a_\ell, b_\ell, \tilde{a}_{\ell -1}, \tilde{b}_{\ell-1}, \}\cup\{\b_i:0 \leq i \leq h\}$.
If $e_i=0$ for all $i$, this result follows from the definition of $\tilde{e}_i$. 
Suppose otherwise, and we can see that $\tilde{e}_{i}$ is defined by $(c_{i}:0\leq i\leq h)$, $(e_{i}:0\leq i\leq h-1)$, and $\{a_\ell, b_\ell, \tilde{b}_{\ell-1}, \tilde{a}_{\ell-1}\}\cup \{\b_i:0 \leq i \leq h\}$.
By assumption on $e_i$, we can see that $(\tilde{e}_{i}:0\leq i\leq h-1)$ is differentially algebraic over $(c_{i}:0\leq i\leq h)$ and $\{a_\ell, b_\ell, \tilde{b}_{\ell-1}, \tilde{a}_{\ell-1}\}\cup \{\b_i:0 \leq i \leq h\}$.
We have already shown that, $(c_{i}:0\leq i\leq h)$ and $(\tilde{c}_{i}:0\leq i\leq h)$ are inter-differentially algebraic, so $(\tilde{e}_{i}:0\leq i\leq h-1)$ is differentially algebraic over $(\tilde{c}_{i}:0\leq i\leq h)$, and $\{a_\ell, b_\ell, \tilde{b}_{\ell-1}, \tilde{a}_{\ell-1}\}\cup \{\b_i:0 \leq i \leq h\}$.
\end{proof}

\textbf{Algorithm B:}

We begin with a system of equations of the form
\begin{equation}\label{eq:algbstart}
    \left\{
\begin{array}{lcl}
\sum_{i=0}^{h} a_i u^{(i)} & = & \sum_{i=0}^{h} b_i z^{(i)} \\
\sum_{i=0}^h c_i u^{(i)}   & = & \sum_{i=0}^h \b_i w^{(i)} \\
\end{array}
\right.
\end{equation}

Algorithm B will eliminate the top equation in this system by first applying Algorithm A $h$ times, reducing the order of the top equation to zero:
\[
\left\{
\begin{array}{lcl}
\tilde{a}_0 u & = & \tilde{b}_0 z^{(i)} \\
\sum_{i=0}^h \tilde{c}_i u^{(i)} + \sum_{i=0}^{h-1} \tilde{e}_i z^{(i)}  & = & \sum_{i=0}^h \b_i w^{(i)} \\
\end{array}
\right.
\]
Finally, $z=\frac{\tilde{a}_0}{\tilde{b}_0}u$, so we substitute this expression for $z$ in the lower equation, eliminating $z$ from the system:
\begin{equation}\label{eq:algbend}
    \sum_{i=0}^h \hat{c}_i u^{(i)}   =  \sum_{i=0}^h \b_i w^{(i)}
\end{equation}
where
\begin{align*}
    \hat{c}_{h}&:=\tilde{c}_{h}\\
    \hat{c}_{i}&:=\tilde{c}_{i}+\sum_{k=0}^{h-i} \binom{h-1}{h-i-k}\tilde{e}_{h-1-k}\left(\frac{\tilde{a}_{0}}{\tilde{b}_{0}}\right)^{(h-i-k)}+ \cdots + \tilde{e}_{i}\left(\frac{\tilde{a}_{0}}{\tilde{b}_{0}}\right).
\end{align*}

\begin{lemma}\label{lem:algb}
For $a_i$, $b_i$, $c_i$, $\hat{c}_i$, and $\b_i$ as defined in Algorithm B, suppose $\{a_i, b_i, \b_i:0\leq i \leq h\}$ is an independent set of differential transendentals and suppose $(a_i:0\leq i \leq h)$ and $(c_i:0\leq i \leq h)$ are inter-differentially algebraic over the other coefficients $\{b_i, \b_i:0\leq i \leq h\}$.
Then $(c_i:0\leq i \leq h)$ and $(\hat{c}_i:0\leq i \leq h)$ are inter-differentially algebraic over $\{b_i, \b_i:0\leq i \leq h\}$.
\end{lemma}
\begin{proof}
By Lemma \ref{lem:alga} and Lemma \ref{lem:ABnew}, we see that $(c_i:0\leq i \leq h)$ is inter-differentially algebraic with $(\tilde{c}_i:0\leq i \leq h)$ over $\{a_i,b_i:0\leq i \leq h\}\cup \{\b_i:0\leq i \leq \ell\}$.
Since $(a_i:0\leq i \leq h )$ is inter-differentially algebraic with $(c_i:0\leq i  \leq h)$ by assumption, it follows that $(c_i:0\leq i \leq h)$ is inter-differentially algebraic with $(\tilde{c}_i:0\leq i \leq h)$ over $\{b_i:0\leq i \leq h\}\cup \{\b_i:0\leq i \leq \ell\}$. Consequently, $(\tilde{c}_i:0\leq i \leq h)$ is independent over $\{b_i:0\leq i \leq h\}\cup \{\b_i:0\leq i \leq \ell\}$.

It follows from the definition that  $(\tilde{c}_{i}:0\leq i \leq h)$ is interdefinable with $(\hat{c}_{i}:0\leq i \leq h)$ over $\{a_{i}, b_{i}:0\leq i\leq h, \}\cup \{\b_i:0 \leq i \leq h\}\cup \{\tilde{e}_{i}:0\leq i\leq h-1\}$. By Lemma \ref{lem:alga}, $(\tilde{e}_{i}:0\leq i\leq h-1)$ is differentially algebraic over $\{\tilde{c}_i:0\leq i \leq h\}\cup\{a_i,b_i:0\leq i \leq h\}\cup\{\b_i:0 \leq i \leq h\}$, and we have already seen that $(a_i:0\leq i \leq h)$ is inter-differentially algebraic with $(\tilde{c}_i:0\leq i \leq h)$. Thus, $(\hat{c}_{i}:0\leq i \leq h)$ is interdefinable with $(\tilde{c}_{i}:0\leq i \leq h)$, and hence with $(c_i:0\leq i \leq h)$ over $\{b_i:0\leq i \leq h\}\cup \{\b_i:0\leq i \leq \ell\}$.
\end{proof}
\begin{remark}
Due to the nature of the substitutions defining Algorithm A and Algorithm B, only coefficients from the first equation influence the lower equation, and not vice versa. For simplicity, Algorithm A and Algorithm B have been defined for systems of two equations, although due to this lack of interaction, these algorithms can be easily applied to systems of many equations. The only change that need be made is that an analogous form of the middle substitution in Algorithm A must be applied to all lower equations simultaneously (lowering the order of $z$ in all of these equations by 1).
\end{remark}
Using Lemma \ref{lem:algb}, we can apply Algorithm B recursively to system \ref{eq:initialsys} until all but one variables have been eliminated. This results in a definable bijection between system \ref{eq:initialsys} and $\AAA^1(\UU)$, resulting in a proof of Theorem \ref{thm:defnbijection}.

\begin{proof}[Proof of Theorem \ref{thm:main}]
Let $f(x)$ be a generic differential polynomial of order $h>1$ and degree $d$. Let $\bar{a}$ be the coefficients of $f$ and let $p$ be the type of a generic solution to $f(x)=0$. Suppose that $d\geq 2 \cdot (\nmdeg(p)+1)$. Let $\bar{a}=(a_1,\ldots a_m)$ be realizations of $p$ which are algebraically independent over $\QQQ\langle \bar{\a}\rangle$ where $m\leq \frac{d}{2}$. By Lemma \ref{lem:higherordergeneric}, the differential tangent space $T^\D_{(\bar{a},\bar{\a})}\left(V^{[m]}\right)$ has coefficients which are independent differential transcendentals over $\QQQ$. This differential tangent space satisfies the hypotheses of Theorem \ref{thm:defnbijection}, so the solution set of  $T^\D_{(\bar{a},\bar{\a})}\left(V^{[m]}\right)$ has no infinite rank proper $\CC$-vector subspaces definable over $\QQQ\langle \bar{a},\bar{\a}\rangle$. By Theorem \ref{bdlinnonlin}, we conclude that the differential variety defined by $f(x)=0$ is strongly minimal.
\end{proof}

\bigskip

Finally, we will present an example of a system of equations which does not have transcendental coefficients where the substitutions of algorithms A and B are not well-defined.
\begin{example}
Consider a generic fiber of the differential equation satisfied by the $j$-function
\[
E_\a:= S(x)+R(x)(x')^2=\a
\]
where $\a$ is a differential transcendental, $S(x)$ is the Schwarzian derivative:
\[
S(x)= \left( \frac{x''}{x'} \right)'-\frac{1}{2}\left( \frac{x''}{x'} \right)^2
\]
and
\[
R(x)=\frac{x^2-1978x+2654208}{2x^2(x-1728)^2}.
\]

Following the procedure laid our in Section \ref{Lintrans}, let $(a_1,a_2)$ be a Morley sequence\footnote{If we were strictly applying the arguments from Section \ref{Lintrans}, we would need a sequence of longer length. However, only two will be necessary to demonstrate this example.} for $E_\a$.
We replace the transcendental $\a$ with a variable $y$, compute the differential tangent space, and eliminate $y$, resulting in the equation:
\[
A_{3}u^{(3)}+A_{2}u''+A_{1}u'+A_{0}u = B_{3}v^{(3)}+B_{2}v''+B_{1}v'+B_{0}v
\]
where
\begin{align*}
    A_{3} &:= \frac{1}{a_1'},
    & B_{3} &:= \frac{1}{a_2'},\\
    A_{2} &:= -\frac{3 a_1''}{(a_1')^2},
    & B_{2} &:= -\frac{3 a_2''}{(a_2')^2},\\
    A_{1} &:= -\frac{a_1^{(3)}}{(a_1')^2}+\frac{3(a_1'')^2}{(a_1')^3} +2a_1'R(a_1),
    & B_{1} &:= -\frac{a_2^{(3)}}{(a_2')^2}+\frac{3(a_2'')^2}{(a_2')^3} +2a_2'R(a_2),\\
    A_{0} &:=(a_1')^2 R'(a_1),
    & B_{0} &:=(a_2')^2 R'(a_2).
\end{align*}

Unlike in Section \ref{genstrongmin}, these coefficients are not independent differential transcendentals. In fact, $A_{2}=3A_{3}'$ and $B_{2}=3B_{3}'$. 
We now apply the sequence of substitutions found in Algorithm A.
First $(u,v)\mapsto (\tilde{u},v)$ where $\tilde{u}=u-\frac{B_{3}}{A_{3}}v$ resulting in the equation
\[
A_{3}\tilde{u}^{(3)}+A_{2}\tilde{u}''+A_{1}\tilde{u}'+A_{0}\tilde{u} = \tilde{B}_2v''+\tilde{B}_1v'+\tilde{B}_0v
\]
where
\begin{align*}
    \tilde{B}_2 &:= B_{2} - 3 A_{3}\left(\frac{B_{3}}{A_{3}}\right)'-A_{2}\left(\frac{B_{3}}{A_{3}}\right),\\
    \tilde{B}_1 &:= B_{1} - 3 A_{3}\left(\frac{B_{3}}{A_{3}}\right)''-3A_{2}\left(\frac{B_{3}}{A_{3}}\right)'-A_{1}\left(\frac{B_{3}}{A_{3}}\right),\\
    \tilde{B}_0 &:= B_{0} -  A_{3}\left(\frac{B_{3}}{A_{3}}\right)^{(3)}-3A_{2}\left(\frac{B_{3}}{A_{3}}\right)''-3A_{1}\left(\frac{B_{3}}{A_{3}}\right)'-A_{0}\left(\frac{B_{3}}{A_{3}}\right).
\end{align*}

The next substitution would replace $v$ with $\tilde{v}$ where $v=\tilde{v}+\left( \frac{A_3}{\tilde{B}_2} \right)\tilde{u}'$ which would reduce the order of the $\tilde{u}$. However, this substitution is not well-defined because $\tilde{B}_2=0$:
\begin{align*}
    \tilde{B}_2 &= 3B_{3}' - 3A_{3}\left(\frac{B_{3}}{A_{3}}\right)'-3A_{3}'\left(\frac{B_{3}}{A_{3}}\right) \\
    &=3B_{3}' - 3A_{3}\left(\frac{B_{3}'A_{3}-A_{3}'B_{3}}{A_{3}^2}\right)-\frac{3A_{3}'B_{3}}{A_{3}} \\
    &= 3B_{3}'-3B_{3}'+ \frac{3A_{3}'B_{3}}{A_{3}} - \frac{3A_{3}'B_{3}}{A_{3}} \\
    &= 0.
\end{align*}
While this specific substitution does not yield a definable bijection, one can use another substitution $\left(\tilde{v}=v-\frac{A_{3}}{\tilde{B}_{1}}\tilde{u}''\right)$, and in this case, the fact that $\tilde{B}_2=0$ results in fewer substitutions necessary to reach the desired result.
\end{example}

\section{Orthogonality to the constants} \label{genstrongmin1}
We've seen that bounds of \cite{freitagmoosa} on the degree of nonminimality can be used in conjunction with linearization techniques to establish the strong minimality of general classes of differential equations. There are two main obstacles to the wide application of these techniques for establishing strong minimality of many classical nonlinear equations. 
\begin{enumerate} 
\item The methods developed in the previous sections seem to require at least one coefficient of the equation in question to be differentially transcendental. 
\item For a given equation, even of small order, the computations required to verify strong minimality are quite involved. 
\end{enumerate} 

In this section, we will show how the computational demands can be significantly reduced if a weaker condition than strong minimality is the goal: \emph{$V$ is either strongly minimal or almost internal to the constants}. In \cite{freitagmoosa}, it is shown that if $\nmdeg(p) > 1,$ then $p$ is an isolated type which is almost internal to a non-locally modular type. In differentially closed fields, this means that $p$ is almost internal to the constant field whenever $\nmdeg(p) > 1.$ Thus, verifying the weaker condition only requires an analysis of forking extensions over one new solution - the computations need only involve two unknown variables in this case. 


\begin{example}
We will show the equation $$x''+x^2-\a =0,$$ where $\a$ is a differential transcendental, is either strongly minimal or internal to the constants. If the equation is not internal to the constants and not strongly minimal, then by the results of \cite{freitagmoosa} there is an indiscernible sequence of length two $(x_1,x_2)$, which would satisfy the system 
$$
\left\{
\begin{array}{lcl}
x_1''+x_1^2 & = & \a \\
x_2''+x_2^2 & = & \a
\end{array}
\right.
$$
such that $x_2$ satisfies an order one equation over $x_1.$ Using the same strategy as in the previous sections, we replace $\a$ with a variable $y$ in both equations, and then compute the differential tangent space:
$$
\left\{
\begin{array}{lcl}
u''+2 x_1 u & = & y \\
v''+2 x_2 v & = & y.
\end{array}
\right.
$$
Eliminating $y$, we are left with the single equation
$$u''+2 x_1 u =v''+2 x_2 v.$$
Consider the definable bijection given by the substitution $(u,v)\mapsto (w,v)$ where $u=w+v$.
This transforms the above equation into 
$$w''+2x_1 w = 2(x_2-x_1) v$$
which can be solved for $v$ since $x_1\neq x_2$.
Therefore the differential tangent space has no infinite rank subvarieties, a contradiction, so $x''+x^2-\a =0$ is either strongly minimal or almost internal to the constants.
\end{example}

\begin{ques} \label{q1}
Varieties which are internal to the constants have certain stronger properties that may, in general, allow one show via some additional argument the strong minimality of specific equations. For instance, by \cite{freitag2017finiteness}, if $X$ is nonorthogonal to the constants, then there are infinitely many co-order one subvarieties of $X$. Thus, in this setting, showing strong minimality (after an argument like that of the example above) is equivalent to ruling out co-order one subvarieties. Are there interesting classes of equations in which one can successfully employ this strategy? 
\end{ques} 


After the completion of this work, the authors, together with Guy Casale and Joel Nagloo, were able to give an affirmative answer to Question \ref{q1} in the case of the equation satisfied by the $j$-function.
This gives a new proof of the main theorem of \cite{freitag2017strong}: the differential equation satisfied by the $j$-function, 
$$\left(\frac{y''}{y'}\right)' -\frac{1}{2}\left(\frac{y''}{y'}\right)^2 + (y')^2 \cdot \frac{y^2-1968y+2654208}{y^2(y-1728)^2} = 0$$
is strongly minimal. We also employ the strategy to establish the strong minimality of several new equations. Nevertheless, we have left Question \ref{q1} as stated above, since we feel pursuing this approach is an important direction for future research.

During the revisions of this manuscript, the bounds on the degree of nonminimality in differentially closed fields were improved dramatically \cite{freitag2022degree} - the degree of nonminimality in differentially closed fields is at most two. This does not affect the statement of our main result, Theorem \ref{thm:main}, but does widen the scope of the result as pointed out in \cite{freitag2022degree} - generic differential equations of degree at least six are strongly minimal.

\bibliography{research}{}
\bibliographystyle{plain}

\end{document}